\documentclass[12 pt,twoside]{amsart}
\usepackage{amsmath}
\usepackage{amssymb}
\usepackage{hyperref}
 \usepackage{amsthm}
 \usepackage{parskip}

\newcommand{\ds}{\displaystyle}
\newcommand{\R}{\mathbb{R}}
\newcommand{\N}{\mathbb{N}}
\newcommand{\C}{\mathbb{C}}
\newcommand{\Z}{\mathbb{Z}}

\newcommand{\clb}{\color{blue}}

\newcommand{\norm}[1]{\left \lVert #1 \right \rVert}
\newcommand{\be}{\begin{equation}}
\newcommand{\ee}{\end{equation}}
\newcommand{\bes}{\begin{equation*}}
\newcommand{\ees}{\end{equation*}}

\newcommand{\ve}{\varepsilon}

\newcommand{\ov}{\overline}

\newcommand{\lam}{\lambda}

\newcommand{\ben}{\begin{enumerate}}
\newcommand{\een}{\end{enumerate}}

\newcommand{\bpf}{\begin{proof}}
\newcommand{\epf}{\end{proof}}

\newcommand{\bsi}{\begin{itemize}[label={{\clb\ding{51}}}]}
\newcommand{\ei}{\end{itemize}}
\newcommand{\bci}{\begin{itemize}[label={{\clb\ding{43}}}]}

\newcommand{\bdeg}{\text{deg}}
\renewcommand{\phi}{\varphi}
\newcommand{\seqn}[1]{(#1)_{n\in \N}}
\numberwithin{equation}{section}
\newtheorem{theorem}{Theorem}[section]

\newtheorem{proposition}[theorem]{Proposition}
\newtheorem{lemma}[theorem]{Lemma}
\theoremstyle{definition}

\theoremstyle{remark}
\newtheorem{remark}[theorem]{\bf Remark}

\begin{document}
\setlength{\parskip}{.15cm}
\title[On the solvability of DNLS]{On the solvability of the discrete nonlinear Schr\"{o}dinger equation with subcubic potential}

\address{Daniel Maroncelli \newline
Department of Mathematics, College of Charleston \newline
66 George Street, Charleston, SC 29424 USA}

\author{Daniel Maroncelli}
% corresponding author
\email{maroncellidm@cofc.edu}

\begin{abstract}
In this paper, we analyze the solvability of the discrete nonlinear Schr\"odinger equation
\bes
i\beta(\Delta_t+\nabla_t)\phi(t,k)
+\gamma |\phi(t,k)|^2\phi(t,k)
+\ve \Delta_k^2\phi(t,k-1)
=
g(t,\phi(t,k)),
\ees
where $\Delta_t$ and $\Delta_k$ denote the standard forward difference operators in the variables $t$ and $k$, respectively, $\nabla_t$ denotes the standard backward difference operator in $t$, {and}
\bes
\Delta_k^2\phi(t,k-1)
=
\phi(t,k+1)-2\phi(t,k)+\phi(t,k-1)
\ees
is the discrete Laplacian operator in the spatial variable $k$. Throughout, we will assume the parameters $\beta$ and $\ve$ are positive real numbers, the parameter $\gamma$ is a nonzero real number, and the potential function $g:\Z\times\C\to \C$ is continuous.
\end{abstract}


\maketitle

\section{Introduction}

In this short work, we analyze the solvability of the discrete nonlinear  Schr\"{o}dinger equation (DNLS)
%\be\label{dschro}\footnotesize
%i\beta(\phi(t+1,k)-\phi(t-1,k)) +\gamma|\phi(t,k)|^2\phi(t,k)+\ve(\phi(t,k+1)-2\phi(t,k)+\phi(t,k-1))=g(\phi(t,k)).
%\ee
\be\label{dschro}
i\beta(\Delta_t+\nabla_t)\phi(t,k) +\gamma|\phi(t,k)|^2\phi(t,k)+\ve\Delta^2_k\phi(t,k-1)=g(t,\phi(t,k)),
\ee
posed on a periodic lattice $\Z_T\times \Z_K$, $T,K\in \N$ with $T\geq 1$ and $K\geq 2$. In particular, we establish the existence of $(T,K)$-periodic solutions to \eqref{dschro} under a very mild growth assumption on the nonlinear potential $g$.  Here $i$ denotes the complex imaginary unit, $(t,k)\in \Z\times \Z$, the parameters $\beta$ and $\ve$ are positive real numbers, the parameter $\gamma$ is a nonzero real number, $\Delta_t$ and $\Delta_k$ are the standard forward difference operators in the variables $t$ and $k$,  respectively, $\nabla_t$ is the standard backwards difference operator in the variable $t$, and the potential function $g:\Z\times \C\to \C$ is continuous.

In the continuous one-dimensional setting, the nonlinear Schr\"{o}dinger equation (NLS) with nonlinear potential is commonly given by
\be\label{cschro}
i\phi_t+\alpha|\phi|^2\phi+\phi_{xx}=g(\cdot, \phi),
\ee
where $\alpha$ is some nonzero real constant.
The literature on the NLS equation is vast. Applications arise in nonlinear optics (where $\phi$ represents a slowly varying envelope of an electromagnetic wave), in water waves and Bose–Einstein condensates (where $\phi$ is a wave function), and in many other areas of nonlinear dispersive wave theory. Mathematically, the model exhibits a very rich combination of dispersive smoothing and nonlinear focusing/defocusing effects, which makes it of interest to mathematicians and physicists alike.

{The discrete nonlinear Schr\"odinger equation (DNLS) considered here, equation $\eqref{dschro}$, arises as a discrete analogue of the one-dimensional continuous nonlinear Schr\"odinger equation $\eqref{cschro}$, obtained by approximating spatial and temporal derivatives with finite difference operators on a uniform lattice. In this way, the DNLS serves as a natural discretization of the continuous model, in which differential operators are replaced by forward and backward difference quotients.}

The existing literature on DNLS is also vast, where it has been examined in both mathematical and physical contexts.  For general theory, see \cite{AblowitzLadik,HennigTsironis,Kevrekidis}. {For recent work regarding the Cauchy initial value problem with general initial data, see} \cite{Vuoksenmaa2024}.  For work with particular emphasis on spatially localized solutions such as breathers, standing waves, and discrete solitons, see \cite{Brunhuber,FlachWillis,Flach2003,Zhang2009} and the references therein. Existence results for periodic and standing wave solutions have been obtained using a variety of methods, see \cite{Bak2010,ZhangMa} for results of this type. Topological degree and fixed point techniques remain fundamental tools in the analysis of nonlinear discrete systems; for results that are similar in nature to this work, we suggest \cite{BereanuMawhin, MaroncelliDiscreteBounded, marrod,MaroncelliMult,Rouche}.

In contrast to much of the existing literature on DNLS, which relies on variational methods or spectral analysis of linearized operators, we adopt a finite-dimensional operator-theoretic framework combined with topological degree theory to establish existence of periodic solutions. We reformulate \eqref{dschro} as a nonlinear operator equation on a Banach space of lattice-periodic functions and exploit compactness properties inherent in the finite-dimensional setting. This approach yields existence results under mild growth conditions on the nonlinear forcing term, without requiring coercivity or a variational structure.  {In this regard, our results complement those works cited above, but extend the class of nonlinearities for which existence can be obtained and provide an alternative theoretical framework based on topological degree methods in a finite-dimensional setting.}

This work is organized as follows. In section \ref{Main},  we prove the existence of periodic solutions to \eqref{dschro} under a very mild subcubic growth condition placed on $g$, that is, when $g$ satisfies $\lim_{|z|\to \infty}\frac{|g(t, z)|}{|z|^3}=0$ for every $t\in \Z$.   {It is important note that the subcubic growth condition imposed here is considerably weaker than the standard assumptions used in topological degree arguments, which typically require sublinear growth of the nonlinearity.} After proving the existence of periodic solutions in this general setting, we point out how our result produces periodic steady-state solutions, which are of interest in applications to the long-term dynamics of the DNLS system. We finish by giving a few concluding remarks.

\section{Main Results}\label{Main}
\subsection{Periodic Solutions}

In this section, we prove the existence of periodic solutions to \eqref{dschro}, under the assumption that $g$ is periodic in its first component. We intend to study \eqref{dschro} as an operator equation defined between Banach spaces.   To facilitate this aim, let us start by introducing some terminology.

%For $n,m\in \N$, with $n\geq m$,  let $[m,n]=\{m,m+1,\cdots, n\}$.  
Fix $T, K\in \N$, with $T\geq 1$ and $K\geq 2$, and define
\bes
X_{T,K}=\{\phi:\Z\times \Z\to \C\mid \phi \text{ is } (T,K)\text{-periodic}\},
\ees
where, as usual, $\Z$ denotes the integers.

\begin{remark}
In this setting, by $(T,K)$-periodic we mean a function $u:\Z\times \Z\to \C$ such that $u(t+T,k)=u(t,k)$  and $u(t, k+K)=u(t,k)$ for all $t,k\in \Z$.
\end{remark}

We make $X_{T,K}$ a Banach space by equipping it with the supremum norm; that is, for $\phi\in X_{T,K}$, we set
\bes
\norm{\phi}=\sup_{(t,k)\in \Z\times \Z}|\phi(t,k)|,
\ees
where $|\cdot|$ denotes the Euclidean norm of $\C$. That $\norm{\phi} <\infty$ for all $\phi\in X_{T,K}$ follows easily from the periodicity requirement.  It also is easy to see that the following holds:

\begin{lemma}
The space $X_{T,K}$ is isomorphic to $\C^{TK}$, in particular, it is a finite-dimensional normed space, and, thus, a Banach space.
\end{lemma}

\begin{proof}
Each $\phi\in X_{T,K}$ is uniquely determined by its values on the finite set
\bes
\{0,\dots,T-1\}\times \{0,\dots,K-1\},
\ees
so the map
\bes
\phi\mapsto (\phi(t,k))_{0\leq t\leq T-1,\;0\leq k\leq K-1}
\ees
defines an isomorphism from $X_{T,K}$ onto $\C^{TK}$.
\end{proof}
 
To study equation $\eqref{dschro}$ in operator form, we introduce the following difference operators on $X_{T,K}$. We define the forward and backward time difference operators $\Delta_t,\nabla_t : X_{T,K}\to X_{T,K}$ by
\be
\begin{aligned}
(\Delta_t \phi)(t,k) &= \phi(t+1,k)-\phi(t,k), \\
(\nabla_t \phi)(t,k) &= \phi(t,k)-\phi(t-1,k),
\end{aligned}
\ee
and the forward spatial difference operator $\Delta_k : X_{T,K}\to X_{T,K}$ by
\begin{equation}
(\Delta_k \phi)(t,k) = \phi(t,k+1)-\phi(t,k).
\end{equation}
We now define operators
 $L\colon X_{T,K}\to X_{T,K}$  by
\be\label{linop}
(L\phi)(t,k)=i\beta(\Delta_t+\nabla_t)(\phi)(t,k)+\ve(\Delta^2_k\phi)(t,k-1),
\ee
and
 $F\colon X_{T,K}\to X_{T,K}$  by
\be\label{cubicop}
(F(\phi))(t,k)=-\gamma|\phi(t,k)|^2\phi(t,k).
\ee
\begin{remark}
I would like to point out that
\bes
\Delta_k^2\phi(t,k-1)
=
\phi(t,k+1)-2\phi(t,k)+\phi(t,k-1)=\Delta\phi(t,k),
\ees
where $\Delta$ denotes the discrete Laplacian operator in the spacial variable $k$.  I have chosen to forgo the Laplacian notation in this setting so that no confusion may arise with the standard notation for forward difference operators.
\end{remark}

For the remainder of this work, we will assume that $g$ is $T$-periodic in its first component; that is, if $g(t+T,x)=g(t,x)$ for all $(t,x)\in \Z\times \C$.  Under this assumption, we may also define 
 $G\colon X_{T,K}\to X_{T,K}$  by
\be\label{nemop}
(G(\phi))(t,k)=g(t,\phi(t,k)).
\ee

The linear operator $L$ captures time differences and spatial coupling at nodes of the $\Z\times \Z$ lattice, the nonlinear operator 
$F$ represents the intrinsic local cubic nonlinearity, and the forcing operator $G$ represents externally applied or prescribed forces, such as time-dependent external driving forces. Finding a $(T,K)$-periodic solution to equation \eqref{dschro} on $\Z\times \Z$ is now equivalent to solving the operator equation 
\be\label{banschro}
L\phi=F(\phi)+G(\phi).
\ee

Since $L$ is a linear operator and $X_{T,K}$ is finite-dimensional Banach space,  $L$ is continuous.  Both $F$ and $G$ are also continuous.
\begin{theorem}
The operators $L$, $F$, and $G$ (see \eqref{linop}, \eqref{cubicop}, and \eqref{nemop}) are continuous. 
\end{theorem}

\begin{proof}
As mentioned already, $L$ is a linear mapping on a finite-dimensional Banach space and so is continuous from a basic fact about linear operators on finite-dimensional topological vector spaces. See \cite{ConwayCA} for reference.

To see that $F$ is continuous, let $h:\C\to \C$ be defined by $h(z)=|z|^2z$.  Certainly $h$ is continuous.  Assume $\seqn{\phi_n}$ is a sequence from $X_{T,K}$ with $\phi_n\to \phi$, for some $\phi\in X_{T,K}$.  Since the convergence is uniform, the collection $\{\phi, \phi_1, \phi_2, \dots\}$ is uniformly bounded.

Let $\phi_0$ denote $\phi$ and choose $M>0$ such that for all $n\in \N_0:=\{0, 1, \dots\}$, $\norm{\phi_n}\leq M$. Let $B_M=\{z\in \C\mid |z|\leq M\}$.  Since $h$ is continuous and $B_M$ is compact, $h$ is uniformly continuous on $B_M$.

Let $\ve>0$. By the uniform continuity of $h$ on $B_M$, we may choose a $\delta>0$ such that if $z, w\in B_M$, with $|z-w|<\delta$, then $|h(z)-h(w)|<\ve$.  Since $\phi_n\to \phi=\phi_0$, we may choose an $N\in \N$ such that if $n\geq N$, then $\norm{\phi_n-\phi_0}<\delta$.  Of course it then follows that if $n\geq N$, we have $\norm{h(\phi_n)-h(\phi_0)}<\ve$.  Thus, $h(\phi_n)\to h(\phi_0)=h(\phi)$ in $X_{T,K}$. However, if $\psi\in X_{T,K}$, then $F(\psi)=-\gamma h(\psi)$ and so we have proved that $F$ is continuous.  

The argument for the continuity of $G$ is similar and so will not be repeated.
\end{proof}

\subsubsection{Brouwer Degree Theory} 
We now come to our main result on the existence of periodic solutions to \eqref{dschro}. We aim to prove the existence of solutions to \eqref{dschro} using a Brouwer degree argument.  For the convenience of the reader, we list those properties of Brouwer degree that are needed in our proof.  

Suppose that $n\in \N$ and that $f:\ov{\Omega}\subset \R^n\to\R^n$, where $\Omega$ is open and bounded.  Suppose further that $f$ is continuous with $f(y)\neq 0$ for all $y\in \partial \Omega$.  As is usual, $\ov{\Omega}$ denotes the closure of $\Omega$, and $\partial \Omega$ its boundary. Under the above assumptions,  there exists a degree function, $\bdeg(f,\Omega, 0)$, with the following properties:
\ben
\item[(a)] $\bdeg(f,\Omega,0)$ is integer-valued;
\item[(b)] if $\bdeg(f,\Omega,0)\neq 0$, then $f(x)=0$ for some $x\in\Omega$;
%\item if $A$ is an invertible matrix and $0\in\Omega$, then $\bdeg(A,\Omega,0)=\text{sgn}(\text{det}(A))$;
\item[(c)] ({\em Homotopy Invariance})  if $H:\ov{\Omega}\times[0,1]\to\R^n$ is continuous, with $H(y,t)\neq0$ for all $y\in \partial \Omega$ and every $t\in[0,1]$, then $\bdeg(H(\cdot,t),\Omega,0)$ is independent of $t\in [0,1]$; 
\item[(d)] ({\em Borsuk's Theorem}) if $\Omega$ is a symmetric neighborhood of $0$, $f$ is odd, and $f(y)\neq 0$ for all $y\in \partial \Omega$,  then $\bdeg(f,\Omega, 0)$ is an odd integer. 
\een

Fixed point arguments and their close relatives, namely Brouwer and Leray-Schauder topological degree theory, have been used extensively over the past several decades to establish existence results for differential and difference equations. For readers interested in a more comprehensive discussion of Brouwer degree, or more generally Leray-Schauder degree, we suggest \cite{Deimling,Rouche}.

\begin{theorem}\label{dnlsperiodic}
Suppose that $T,K\in \N$, $T\geq 1$ and $K\geq 2$. Suppose that $g$ is $T$-periodic in the variable $t$.  If $\lim_{|z|\to \infty} \frac{g(t,z)}{|z|^3}=0$, for each $t\in \Z$, then \eqref{dschro} has a $(T,K)$-periodic solution.
\end{theorem}

\begin{proof}

We know that solving \eqref{dschro} is equivalent to solving \eqref{banschro}, that is, solving \eqref{dschro} is equivalent to solving 
\be\label{oper2}
L\phi=F(\phi)+G(\phi).
\ee
Let $\ds\norm{L}=\sup_{\psi\in X_{T,K}, \norm{\psi}=1}\norm{L\psi}$; that is, let $\norm{L}$ denote the operator norm of $L$.

Choose a real number $s$ with $s>\norm{L}$. Let $\sigma(L)$ denote the spectrum of $L$, that is, 
\bes
\sigma(L)=\{\lam\in \C\mid \lam I-L \text{ is not invertible}\}.
\ees
If $s>\norm{L}$, then it is well-known (from the theory of Banach Algebras) that $s\not\in \sigma(L)$. Thus,  $L-sI$ is an invertible linear mapping.  Let $A=L-sI$. Define {$P=A^{-1}F$} and $H=A^{-1}(G-sI)$. Solving \eqref{oper2} is now equivalent to solving 
\be
\phi=P(\phi)+H(\phi).
\ee

Set $S=I-P$.  Note that since $F$ is odd, so is $S$.  If $R>0$ is large, then $S \phi\neq 0$ whenever $\norm{\phi}=R$.  Indeed, if $S\phi=0$, then {$\phi=P\phi$}, or, equivalently, $A\phi=F\phi$.  However, 
\bes
\norm{F(\phi)}=|\gamma|\norm{\phi}^3=|\gamma|R^3
\ees
 and 
\bes\norm{A\phi}\leq (\norm{L}+s)\norm{\phi}=(\norm{L}+s)R,
\ees
so that equality of $F(\phi)$ and $A\phi$ is impossible when $\phi$ has large norm.
By Borsuk's theorem on Brouwer degree (see (d) above), $\bdeg(S,B_R,0)$ is an odd integer for large $R$, where $B_R=\{\psi\in X_{T,K}\mid \norm{\psi}<R\}$.  

Let $Q=S-H$.  Note that $\phi$ solves \eqref{dschro} if and only if $Q(\phi)=0$.  Our goal is to show that $\bdeg(Q,B_R,0)=\deg(S,B_R,0)$, for some appropriately chosen ball $B_R$.  To do this, we will show that 
\bes
\norm{S(\phi)-Q(\phi)}<\norm{S(\phi)}
\ees
or, equivalently,
\be\label{E:Hom}
\norm{A^{-1}(G(\phi)-s\phi)}=\norm{H(\phi)}<\norm{S(\phi)}=\norm{\phi-A^{-1}F(\phi)}
\ee
for every $\phi\in X_{T,K}$ with $\norm{\phi}=R$, for this appropriately chosen $R$.  In this case, for $\phi\in X_{T,K}$ with $\norm{\phi}=R$, we must have that for any $t\in [0,1]$
\bes
\begin{split}
\norm{(1-t)S(\phi)+tQ(\phi)}&=\norm {t(Q(\phi)-S(\phi))+S(\phi)}\\&>\norm{S(\phi)}-t\norm{S(\phi)-Q(\phi)}\\&>0.
\end{split}
\ees
Thus, if we show \eqref{E:Hom}, the result will follow by the invariance of Brouwer degree under homotopy. 

By our assumption on $g$, we may choose $R^*>0$ such that if $|z|\geq R^*$, then for all $t\in \mathbb{Z}$,
\[
|g(t,z)|<\frac{|\gamma|}{2\|A\|\|A^{-1}\|}|z|^3.
\]
Indeed, for each fixed $t\in\mathbb{Z}$, the assumption
\[
\lim_{|z|\to\infty}\frac{g(t,z)}{|z|^3}=0
\]
implies the existence of a corresponding threshold $R_t>0$ with this property. Since $g$ is $T$-periodic in $t$, it suffices to take
\[
R^*=\max_{t\in\{0,1,\dots,T-1\}} R_t,
\]
which yields a uniform bound valid for all $t\in\mathbb{Z}$. 

 If we set $M=\sup_{\{(t,z)\in \Z\times \C, |z|\leq R^*\}}|g(t,z)|$, then it follows that for all $z\in \C$ and any $t\in \Z$, 
\bes
|g(t,z)|\leq M+\frac{|\gamma|}{2\norm{A}\norm{A^{-1}}}|z|^3.
\ees
We now easily deduce, by taking suprema,  that 
\be\label{Subop}
\norm{G(\phi)}\leq M+\frac{|\gamma|}{2\norm{A}\norm{A^{-1}}}\norm{\phi}^3
\ee
for every $\phi\in X_{T,K}$.  Since $X_{T,K}$ is finite-dimensional, $A^{-1}$ is a continuous operator.  As such, we conclude, with the help of \eqref{Subop}, that 
\be\label{sublin}
\norm{A^{-1}(G(\phi)-s\phi)}\leq \norm{A^{-1}}\norm{G(\phi)-s\phi}\leq\norm{A^{-1}}(\norm{G(\phi)}+s\norm{\phi})\leq B+C\norm{\phi}+D\norm{\phi}^3, 
\ee
 for every $\phi\in X_{T,K}$, where $B=\norm{A^{-1}}M$, $C=\norm{A^{-1}}s$, and  $D=\frac{|\gamma|}{2\norm{A}}$.

{By the continuity of $A$, we must have}
\bes
\norm{\psi}=\norm{AA^{-1}\psi}\leq \norm{A}\norm{A^{-1}\psi}
\ees 
or
\be\label{boundedbelow}
\dfrac{1}{\norm{A}}\norm{\psi}\leq \norm{A^{-1}\psi}
\ee
for each $\psi\in X_{T,K}$.  From \eqref{boundedbelow} we deduce, 
\be\label{suplin}
\begin{split}
\norm{\phi-A^{-1}F(\phi)}\geq \norm{A^{-1}F(\phi)}-\norm{\phi}&\geq\frac{1}{\norm{A}}\norm{F(\phi)}-\norm{\phi}\\&=\frac{|\gamma|}{\norm{A}}\norm{\phi}^3-\norm{\phi}\\
						  &=2D\norm{\phi}^3-\norm{\phi}
\end{split}
\ee
{for every $\phi\in X_{T,K}$, where we have used the fact that for each such $\phi$},
\bes
\norm{F(\phi)}=\sup_{(t,k)\in \Z\times \Z}|\gamma||\phi(t,k)|^3=\gamma\norm{\phi}^3.
\ees

Since {$\lim_{p\to\infty}\frac{2Dp^3-p}{B+Cp+Dp^3}=2$}, it follows from \eqref{sublin} and \eqref{suplin} that there exists an $R>0$ such that if $\phi\in X_{T,K}$ with $\norm{\phi}=R$, then 
\bes
\norm{A^{-1}(G(\phi)-s\phi)}<\norm{\phi-A^{-1}F(\phi)},
\ees
which verifies \eqref{E:Hom}. Based on our discussion following \eqref{E:Hom}, the result follows.
 \end{proof}

\subsection{Periodic steady-state solutions}
Theorem \ref{dnlsperiodic} can be used to deduce periodic steady-state solutions to \eqref{dschro}. By steady-state solutions, we mean solutions of the form
\bes
\phi(t,k)=u(k),
\ees
which are independent of the time variable $t$. In this case, the discrete time-difference terms vanish, and \eqref{dschro} ``reduces'' to a purely spatial nonlinear difference equation of the form
\be\label{steadydnls}
\varepsilon \Delta_k^2 u(k-1) + \gamma |u(k)|^2 u(k) = h(u(k)),
\ee
where $\ve>0$, $\gamma\in \R$, and $h:\C\to \C$ is continuous, when we assume our external driving force $g$ is independent of time.

Steady-state solutions play an important role in applications of the discrete nonlinear Schr\"odinger equation, as they correspond to time-independent wave profiles that persist under the dynamics of the system. In physical models, such solutions typically represent equilibria or standing wave configurations in lattices, and are often used to describe localized modes in nonlinear optical arrays, energy transport in coupled oscillator systems, and stationary states in Bose--Einstein condensates.

The following proposition is an immediate corollary of Theorem \ref{dnlsperiodic}. 
\begin{proposition}
Suppose that $K\in \N$, $K\geq 2$.  If $\lim_{|z|\to \infty} \frac{h(z)}{|z|^3}=0$, then \eqref{steadydnls} has a $K$-periodic solution.
\end{proposition}

\begin{proof}
Set $g(t,z)=h(z)$, and apply Theorem \ref{dnlsperiodic} with $T=1$.
 \end{proof}

\medskip

\subsection{Example}
Part of the novelty of theorem \ref{dnlsperiodic} is that it is applicable for arbitrary periods $T\geq 1$, $K\geq 2$ and such a wide variety of nonlinearities $g$.  Additionally, the growth requirement on $g$ is rather mild compared with requirements typical when proving existence results in differential and difference equations using degree theory arguments, where $g$ is typically required to adhere to some sublinear growth condition. 

 It is not hard to think of functions $g$ which satisfy the requirement of Theorem \ref{dnlsperiodic} ($\lim_{|z|\to \infty} \frac{g(t,z)}{|z|^3}=0$, for each $t\in \Z$), so here we merely indicate those functions we had in mind when formulating the theorem.  If $g$ is of the form
\bes
g(t,z) = f(t)z^r,
\ees
where $0<r<3$ and $f$ is $T$-periodic,
then $g$ is certainly $T$-periodic in $t$ and satisfies the growth condition
\bes
\lim_{|z|\to\infty} \frac{g(t,z)}{|z|^3} = 0,
\ees
uniformly in $t$.
Thus, for such a $g$, Theorem \ref{dnlsperiodic} ensures a $(T,K)$-periodic solution to \eqref{dschro} for all $K\geq 2$.

\section*{Concluding Remarks}

The approach developed in this work is based on reducing the DNLS on a finite periodic lattice to a finite-dimensional operator equation and applying Brouwer degree theory. This framework is well-suited to $(T,K)$-periodic boundary conditions, where the underlying function space becomes finite-dimensional and compactness issues are avoided.

It is natural to ask how these methods extend to other boundary conditions, such as discrete Dirichlet, Neumann, and Robin conditions. For such cases, additional care is required since the discrete Laplacian may no longer act invariantly on a ``natural'' function space. 

Another important direction concerns the study of DNLS posed on infinite lattices $\mathbb{Z}\times\mathbb{Z}$ (without periodic conditions). In this setting, the problem becomes genuinely infinite-dimensional, and the finite-dimensional Brouwer degree is no longer applicable. One would instead need to employ infinite-dimensional extensions such as Leray-Schauder degree, fixed point index theory, or compactness methods in weighted sequence spaces. These extensions are significantly more delicate, as superlinear cubic growth in this setting is hard to tame.

\end{document}